\documentclass[11pt, a4paper]{amsart}
\usepackage{amsmath}
\usepackage{amssymb}
\usepackage{amsthm}
\usepackage[utf8]{inputenc}
\usepackage[a4paper, total={6in, 10in}]{geometry}
\usepackage{geometry}

\newcommand*\R{\mathbb{R}}

\newcommand*\rn{\mathbb{R}^n}
\newcommand*\Om{\Omega}
\newcommand*\n{\Vert}
\newcommand*\pde{-\ div\ \big(\langle A\nabla u,\nabla u\rangle^{\frac{p-2}{2}}A\nabla u\big)}
\newcommand*\li{L^q(B_1^+)}
\usepackage[english]{babel}

\usepackage{etoolbox}
\usepackage{hyperref}
\hypersetup{
    colorlinks=true,
    linkcolor=blue,
    filecolor=magenta,      
    urlcolor=cyan,
}
\title[planar $p$-Poisson equation]{\textbf{ A note on the optimal  boundary regularity for the  planar generalized $p-$Poisson equation}}
\author{Saikatul Haque \\TIFR-Centre for Applicable Mathematics}

\newtheorem{theorem}{Theorem}[section]                                      
\newtheorem{corollary}[theorem]{Corollary}

\newtheorem{lemma}[theorem]{Lemma}
\newtheorem{prop}[theorem]{Proposition}

\newtheorem{remark}[theorem]{Remark}

\usepackage[centerlast,small,sc]{caption}
\begin{document}
\maketitle

\begin{abstract}
In this note, we establish  sharp regularity for solutions to the following generalized $p$-Poisson  equation  
$$\pde=-\ div\ \mathbf{h}+f$$ 
in the plane (i.e. in $\R^2$) for $p>2$ in the presence of  Dirichlet  as well as  Neumann boundary conditions and with $\mathbf{h}\in C^{1-2/q}$, $f\in L^q$, $2<q\leq\infty$. The   regularity assumptions on  the principal part $A$ as well as that  on  the   Dirichlet/Neumann conditions are exactly the same   as in the linear case and therefore sharp (see Remark \ref{sharp} below).  Our main results Theorem \ref{main1} and Theorem \ref{main2} should be thought of as  the boundary   analogues     of the sharp interior  regularity   result established in the  recent interesting paper \cite{ATU} in the case of 
\begin{equation}\label{e0}
-\ div\ (|\nabla u|^{p-2} \nabla u) =f
\end{equation}
for more general variable coefficient operators and with an additional divergence term.

\end{abstract}
\section{Introduction}

In this paper, we study  sharp $C^{1,\alpha}$ regularity  estimates  in the plane for 
\begin{equation}\label{e1}
\pde=-\ div\ \mathbf{h}+f,
\end{equation} 
with Dirichlet/Neumann boundary conditions when $\mathbf{h}\in C^{1-2/q}$, $f\in L^q$, $2<q\leq\infty$.  In the linear case, i.e. when $p=2$, it is  well known that solutions to 
\[
-\ \Delta u =f\in L^\infty(B_1)
\]
are of class $C^{1, \alpha}_{loc}$ for every $\alpha < 1$ but need not be in $C^{1,1}$.  In the degenerate  setting $p>2$, the situation is quite   different and the smoothing effect of the operator is  less prominent as the following radially symmetric example shows.  More precisely for ($0<a<1$)
\begin{equation}\label{ex1}
u(x)= |x|^{1+a}
\end{equation}
(as mentioned in \cite{az}) we have that 
\begin{equation}\label{ex}
div\ (|\nabla u|^{p-2} \nabla u)=c_{a,p}|x|^{ap-a-1}
\end{equation} for some constant $c_{a,p}\neq0$ if $a\neq3/(p-1)$. The RHS is in $L^q(B_1)$ if $a>(1-2/q)/(p-1)$. This example shows   that the best  regularity  that one can expect  for solutions to  \eqref{e0} is $C^{1,(1-2/q)/(p-1)}$. In fact, this example gives  rise to the following well known  conjecture among the experts in this field and is referred to  as the $C^{p'}$  conjecture.

\medskip

\textbf{Conjecture}($C^{p'}$ conjecture): Solutions to \eqref{e0} are locally of class $C^{1, \frac{1}{p-1}}$ for $p>2$ with $f\in L^\infty$. 

\medskip
Note that the same example shows if $\mathbf{h}\in C^{1-2/q}$ then the best  regularity we can expect for solutions to \eqref{e1} is $C^{1,(1-2/q)/(p-1)}$.

\medskip

Over here,   we would like to  mention that although the conjecture is open, nevertheless  it is well known  that solutions to \eqref{e0} are  locally of class $C^{1, \alpha}$ for some exponent $\alpha$ depending on $p$ and $n$. See for instance, \cite{Di}, \cite{Le}, \cite{Tk}, \cite{Ur}.

\medskip

Very recently, the $C^{p'}$ conjecture has been solved  in the planar case in \cite{ATU}.  The proof in \cite{ATU}  relies on a crucial global $C^{1, \alpha}$ estimate for $p$-harmonic functions in the planar case for  some $\alpha> \frac{1}{p-1}$  combined with a  certain  geometric  oscillation estimate  which has its roots   in the  seminal paper of Caffarelli, see \cite{ca}.  This very crucial  global  $C^{1, \alpha}$ estimate for planar $p$-harmonic functions follows from results in \cite{BK} which exploits the fact that the complex  gradient of a $p$-harmonic function in the plane is a $K$-quasiregular mapping.  Over here, we would like to mention that no analogous result  concerning similar quantitative  regularity   for $p$-harmonic  functions  is known in higher dimensions. 

\medskip

In \cite{ll}, \eqref{e0} was studied with $f\in L^q$, $2<q<\infty$ and optimal interior regularity was achieved in plane by Lindgren and Lindqvist. 
Recently in an interesting work of Araujo and Zhang, \cite{az} more general $p$-Poisson equation (but $\mathbf{h}=0$) is studied and some interior regularity is achieved.
We assume $A\in C^{(1-2/q)/(p-1)}$ to achieve the same regularity as in the case of the $p$-Poisson equation.
 
\medskip 

In this paper, we  make the observation that the ideas  in \cite{ATU}  can be applied to  more general  variable coefficient  equations   with Dirichlet and   Neumann boundary conditions. Over here,  we would like to mention that although our work has been strongly  motivated by that in \cite{ATU}, it has nonetheless required certain delicate adaptations in our setting due to the presence of the boundary datum. To apply certain iteration, as in  \cite{ATU}, one needs to ensure smallness of boundary datum at each step of iteration, as the reader can see in the proofs of Theorems \ref{dC}, \ref{nC}.  Moreover, we finally needed  to combine the interior estimate and the estimate at the boundary  in order to get a uniform estimate and this required a bit of subtle analysis as well, as can be seen in the proof of Theorem \ref{main1} after Theorem \ref{dB}. In closing, we would like to mention two other interesting results  which are closely related to this article. In \cite{KM1} (see also \cite{KM2}),  Kuusi and Mingione established the continuity of $\nabla u$ assuming $f$ in the Lorentz space $L(n, \frac{1}{p-1})$  and where the principal part is sightly more general  as in \cite{az} and has Dini dependence in $x$. Moreover, a moduli of continuity of $\nabla u$ is also established in the same article when the principal part has H\"older dependence in $x$ and $f \in L^{q}$ for $n < q \leq \infty$. 

\medskip

The paper is organized as follows: In Section 2,  in order to keep our paper  self contained, we gather some known regularity results  and then state  our main results. In Section 3, we  prove our main result by following the ideas in \cite{ATU}. Finally in Section 4 and 5, we prove auxiliary regularity results of Giaquinta-Giusti type (see \cite{gg}) for certain  equations with   quadratic non-linearities  which are required  in our analysis.  Such results hold for    arbitrary dimensional Euclidean space $\rn$ and are slight extensions of the results in \cite{gg} and  hence could possibly be of independent interest. 

\section{Preliminaries and statements of the main results}
For notational convenience, we will denote the nonlinear $p$-laplacian operator by $\Delta_p$. We will denote by $W^{1, p}(O)$ the Sobolev space of functions  $g$ which together with its distributional derivatives $g_{x_i}$, $i=1,\cdots,n$, are $L^{p}$ integrable. Also $\nabla g$ (or sometimes $Dg$) will denote the total gradient of $g$. We will denote by $C^{1, \alpha}(O)$ the class of functions $v$ which have H\"older continuous first order  derivatives  with H\"older  exponent $\alpha$. By $\n\cdot\n_{C^{\alpha}}$ we mean the maximum of $L^\infty$-norm and the $\alpha$-Holder semi-norm. Also by $\n\cdot\n_{C^{1,\alpha}}$ we mean the maximum of $L^\infty$-norm and the $\n\cdot\n_{C^{\alpha}}$-norm of the gradient. Finally, an arbitrary point in $\R^n$ will be denoted by $x=(x_1,\cdots,x_n)$  and $B_r(x)$ will denote  the Euclidean ball in $\R^n$ centered  at $x$ of radius $r$. Let us set $B_r^+=B_r(0)\cap\{x_n>0\}$ and $B_r'=B_r(0)\cap\{x_n=0\}$.  Throughout our discussion, we will always assume that $p>2$. 

\medskip

 We  now  state the relevant result from \cite{ATU} concerning  quantitative $C^{1, \alpha}$ regularity for planar $p$-harmonic functions which is obtained from the estimates in \cite{BK}.  See Proposition 2 in \cite{ATU}.

\begin{theorem}\label{quant}
For any $p >2$, there exists $0 < \tau_0 < \frac{p+2}{p-1}$ such that $p$-harmonic functions in $B_1 \subset \R^2$  are locally of class $C^{1, \frac{1}{p-1} + \tau_0}$. Furthermore, one has the  following quantitative estimate
\begin{equation}\label{q0}
||u||_{C^{1, \frac{1}{p-1}+\tau_0}}(B_{1/2}) \leq C_p ||u||_{L^{\infty}(B_1)}
\end{equation}

\end{theorem}

Before proceeding further, we make the following important remark.

\begin{remark}
We note that in the plane,  infact  there is a better regularity  result due to Iwaniec and Manfredi (see \cite{IM}) which assures that any $p$-harmonic function is of class $C^{1, \alpha^*}$ where  $\alpha^* > \frac{1}{p-1} + \tau_0$ where $\tau_0$ is as in Theorem \ref{quant}. More precisely, we have that $\alpha^*$ has the following explicit expression
\[
\alpha^*= \frac{1}{6} \bigg( \frac{p}{p-1} + \sqrt{1 + \frac{14}{p-1} + \frac{1}{(p-1)^2}}\  \bigg)
\]
and this regularity is optimal. However, no explicit estimates on the control of $C^{1, \alpha^*}$ norm  has been stated in \cite{IM} and that is precisely the reason as to why the proof in \cite{ATU} relies on the estimate in  Theorem \ref{quant}. 
\end{remark}
We now state the relevant $C^{1, \alpha}$ boundary regularity result established in \cite{lieb}. These results hold for arbitrary dimensional Euclidean space $\rn$ and will be needed in the compactness arguments as in the proof of Lemma \ref{dc0} and Lemma \ref{cc0}. 
\begin{prop}\label{ci}
Let $u\in W^{1,p}(B_1^+)$ solves \begin{equation}\label{ch}
-div\ \mathbf{A}(x,Du)=-div\ \mathbf{h}+f \text{ in }B_1^+,\ \ \ u=\phi\text{ on }B_1'
\end{equation}
with $\n u\n_{L^\infty(B_1^+)}\leq1$ and $\mathbf{A}$ satisfies
\begin{equation}\label{cg}
\partial_{\zeta_j}\mathbf{A}^i(x,\zeta)\xi_i\xi_j\geq\lambda|\zeta|^{p-2}|\xi|^2,\ |\partial_{\zeta_j}\mathbf{A}^i(x,\zeta)|\leq L|\zeta|^{p-2},\ |\mathbf{A}(x,\zeta)-\mathbf{A}(y,\zeta)|\leq L(1+|\zeta|)^{p-1}|x-y|^\alpha,
\end{equation}
  $\n\mathbf{h}\n_{C^\alpha(\overline{B_1^+})}\leq L$, $\n f\n_{L^q(B_1^+)}\leq L$ where $q>n$, $\n\phi\n_{C^{1,\alpha}(B_1')}\leq L$. Then there exists a positive constant $\beta=\beta(\alpha,\lambda,L,p,q,n)$ such that $$\n u\n_{C^{1,\beta}(\overline{B_{3/4}^+})}\leq C(\alpha,\lambda,L,p,q,n).$$ 
\end{prop}
\begin{prop}\label{cj}
Let $u\in W^{1,p}(B_1^+)$ solves \begin{equation}
-div\ \mathbf{A}(x,Du)=-div\ \mathbf{h}+f \text{ in }B_1^+,\ \ \ \mathbf{A}^n(x,Du)=\phi\text{ on }B_1'
\end{equation}
with $\n u\n_{L^\infty(B_1^+)}\leq1$ and $\mathbf{A}$ satisfies \eqref{cg}, $\n\mathbf{h}\n_{C^\alpha(\overline{B_1^+})}\leq L$, $\n f\n_{L^q(B_1^+)}\leq L$ where $q>n$, $\n\phi\n_{C^{\alpha}(B_1')}\leq L$. Then there exists a positive constant $\beta=\beta(\alpha,\lambda,L,p,q,n)$ such that $$\n u\n_{C^{1,\beta}(\overline{B_{3/4}^+})}\leq C(\alpha,\lambda,L,p,q,n).$$ 
\end{prop}
\begin{remark}
We note that although the regularity result in \cite{lieb} is stated for $f \in L^{\infty}$ and $\mathbf{h}=0$, nevertheless the proof in \cite{lieb} can be adapted in  a straightforward way to cover the situations in Proposition \ref{ci} and Proposition \ref{cj}. 
\end{remark}

From now on till the end of section 3, we set $\gamma=(1-2/q)/(p-1)$, $\Gamma=1-2/q$. For $q=\infty$ we interpret $1/q$ as zero. We now state  our main results.  In the case of Dirichlet conditions, we have  the following result. 

\begin{theorem}\label{main1}
Let $\Omega$ be a $C^{1, \gamma}$ domain in $\mathbb{R}^2$ and $x_0 \in \partial \Omega$.  Let   $u$ be a solution to 
\begin{equation}\label{E0}
\begin{cases}
\pde=-\ div\ \mathbf{h}+f\ \text{in}\ \Omega \cap B_{r_0}(x_0)
\\
u=g\ \text{on}\ \partial \Omega \cap B_r(x_0)
\end{cases}
\end{equation}
 for some $r_0 > 0$. On the coefficient matrix $A$, we assume that   $A \in C^{\gamma}(\overline{\Omega})$ and  there exists $\lambda>0$ such that
 \begin{equation}\label{A}
 \lambda|\xi|^2\leq A_{i,j}\xi_i\xi_j\leq\lambda^{-1}|\xi|^2
 \end{equation}
Furthermore assume that  $g \in C^{1, \gamma}(\partial \Omega)$, $\mathbf{h}\in C^{1-2/q}(\overline{\Om})$ and $f \in L^q(\Omega)$ with $q>2$. Then we have that $u \in C^{1, \gamma}(\overline{\Omega \cap B_{\frac{r_0}{2}} (x_0)})$. 
\end{theorem}

Now in the Neumann case, our result can be stated as follows.
\begin{theorem}\label{main2}
Let $\Omega$ be a $C^{1, \gamma}$ domain in $\mathbb{R}^2$ and $x_0 \in \partial \Omega$. Denote by  $\nu$  the outward unit normal to $\partial \Omega$.   Let   $u$ be a solution to 
\begin{equation}\label{E1}
\begin{cases}
\pde=-\ div\ \mathbf{h}+f\ \text{in}\ \Omega \cap B_{r_0}(x_0)
\\
\langle A\nabla u, \nabla u\rangle^{\frac{p-2}{2}} \langle A \nabla u, \nu\rangle =|g|^{p-2} g \ \text{on}\ \partial \Omega \cap B_r(x_0)
\end{cases}
\end{equation}
 for some $r_0 > 0$. The assumptions on $A$ are as in the previous theorem. Furthermore assume that $g \in C^{ \gamma}(\partial\Omega)$ , $\mathbf{h}\in C^{1-2/q}(\overline{\Om})$ and $f \in L^q(\Omega)$ with $q>2$. Then  $u \in C^{1, \gamma}(\overline{\Omega \cap B_{\frac{r_0}{2}} (x_0)})$. 
\end{theorem}
\begin{remark}\label{sharp}
Over here we would like to mention that  the regularity result  in  Theorem \ref{main1} and Theorem \ref{main2}  are sharp in the sense that even  in the linear  case $p=2$, we would  get the same conclusion with the stated  regularity  assumptions on $A, \Om$ and $g$. 
\end{remark}

\section{Proof of the main  results}

\subsection{Proof of Theorem \ref{main1}}
 Without loss of generality,  by rotation of coordinates,  we may assume that  there exists $s< r_0$ such that for $x_0\in\partial\Om$, $\Om \cap B_{s}(x_0)$ is given by $\{(x_1, x_2):x_2 > \phi(x_1)\}\cap B_{s}(x_0)$ where $\phi \in C^{1, \gamma}$. It suffices to show that $u \in C^{1, \gamma}(\overline{\Om \cap B_{\frac{s}{2}}(x_0)})$ and then the  conclusion of the theorem would follow by a standard covering argument.  Now by using the transformation,
 \begin{equation}\label{tran}
 \begin{cases}
 y_1=x_1
 \\
 y_2= x_2- g(x_1)
 \end{cases}
 \end{equation}
 we may  assume that $u$ solves 
 \begin{equation}\label{E2}
 \begin{cases}
 \pde=-\ div\ \mathbf{h}+f\ \text{in}\ B_s^+
 \\
 u=g\ \text{on}\ B_{s} \cap \{x_2=0\}
 \end{cases}
 \end{equation}
 where $A, f, g$ satisfies similar assumptions as in the hypothesis of Theorem \ref{main1}. Then by the following change of variables
 \[
 x \mapsto A(0)^{\frac{-1}{2}} x
 \]
 and a subsequent rotation, we can reduce it to the case that $A(0)=I$ where $I$ denote the identity matrix. 
  Now by using the   rescaling  $u_{s}(x)= u(sx)$, we may assume that $s=1$ in \eqref{E2}. We now let $$\mathcal{M}_\lambda=\big\{M\in C^\gamma(B_1^{+}\cup B_1',M_2(\R)):\lambda|\xi|^2\leq M_{i,j}\xi_i\xi_j\leq\lambda^{-1}|\xi|^2, \n M-I\n_{C^{\gamma}(B_1^{+})} \leq 1, M(0)= I\big\},$$
\begin{equation}\label{0}
\mathcal{O}_\lambda=\big\{M\in C^\gamma(B_1,M_2(\R)):\lambda|\xi|^2\leq M_{i,j}\xi_i\xi_j\leq\lambda^{-1}|\xi|^2, \n M-I\n_{C^{\gamma}(B_1)} \leq 1, M(0)= I\big\}
\end{equation}
 Following \cite{ATU},  we denote by $\mathcal{A}(p, d)$ the following class of functions.  $$\mathcal{A}(p,d)=\big\{u\in L^{\infty}(B_1)\cap W^{1,p}(B_{1/2}):\Delta_pu=0\text{ in }B_{1/2}\big\}.$$

We now  have the following boundary version of the small $C^{1}$ corrector lemma (See Lemma 3 in \cite{ATU}). Let $\alpha\geq\gamma$.
\begin{lemma}\label{dc0}
Let $u\in W^{1,p}(B_{1}^{+})$ be a weak solution of $$\pde=-div\ \mathbf{h}+f\text{ in }B_{1}^{+},\ u=g\text{ on }B_{1}',$$ with $\n u\n_{L^\infty(B_{1}^{+})}\leq1$, $A\in\mathcal{M}_\lambda$. Then for given $\varepsilon>0$, there exists $\delta=\delta(p,\varepsilon,\lambda)>0$ such that if $\n \mathbf{h}\n_{C^\alpha(\overline{B^+})}\leq\delta$, $\n f\n_{L^q(B_1^{+})}\leq \delta$, $\n g\n_{C^{1,\gamma}(B_1^{'})}\leq\delta$ and $\n A-I\n_{C^\gamma(B_1^{+})}\leq\delta$, then we can find a corrector $\xi\in C^1(\overline{B_{1/2}^{+}})$, with $$|\xi(x)|\leq\varepsilon\text{ and }|\nabla\xi(x)|\leq\varepsilon\text{ for all }x\in B_{1/2}^{+}$$ satisfying$$-\ \Delta_p(u+\xi)=0\text{ in }B_{1/2}^{+} \text{ and }u+\xi=0\text{ on }B_{1/2}'.$$
\end{lemma}
\begin{proof}
If not there exists an $\varepsilon_0>0$ and sequences $\{u_j\}$ in $ W^{1,p}(B_1^{+})$, $\{\mathbf{h}_j\}$ in $C^\alpha(\overline{B_1^+})$, $\{f_j\}$ in $L^q(B_1^{+})$, $\{g_j\}$ in $C^{1,\gamma}(B_1^{'})$ and $\{A_j\}$ in $\mathcal{M}_\lambda$ satisfying $$-\ div(\langle A_j\nabla u_j,\nabla u_j\rangle^{\frac{p-2}{2}}A_j\nabla u_j)=-div\ \mathbf{h}_j +f_j\text{ in }B_1^+,\ u_j=g_j\text{ on }B_1^{'},$$ $$\n u_j\n_{\li}\leq1,\n \mathbf{h}_j\n_{C^\alpha(\overline{B^+})}\leq\frac{1}{j},\ \n f_j\n_{L^q(B_1^+)}\leq\frac{1}{j},\ \n g_j\n_{C^{1,\gamma}(B_1^{'})}\leq\frac{1}{j}\ \text{and }\n A_j-I\n_{C^\gamma(B_1^{+})}\leq\frac{1}{j}$$ but for each $\xi$ satisfying $$-\ \Delta_p(u_j+\xi)=0\text{ in }B_{1/2}^{+}\text{ and }u_j+\xi=0\text{ on }B_{1/2}^{'}.$$ we have $\n\xi\n_{C^1(B_{1/2}^{+})}>\varepsilon_0$.
Using  Proposition \ref{ci}, we have that for some $\beta>0$ depending on $p,\gamma, \lambda$,
$$\n u_j\n_{C^{1,\beta}(\overline{B_{3/4}^{+}})}\leq C=C(\lambda,p,q,\gamma).$$
Now by applying Arzela-Ascoli, we can assert that for a subsequence $\{u_j\}$,  $u_j\rightarrow u_\infty$ in $C^{1}(\overline{B_{3/4}^+})$ and $u_\infty\in C^{1,\beta}(\overline{B_{3/4}^+})$. Also by  a standard weak type  argument using test functions, we can show that  $$-\ \Delta_pu_\infty=0 \text{ in }B_{1/2}^{+}\text{ and }u_{\infty}=0\ \text{on}\ B_{1/2}^{'}$$ At this point, by setting $\xi_j=u_\infty-u_j$, we see that  $$-\ \Delta_p(u_j+\xi_j)=-\ \Delta_pu_\infty=0 \text{ in }B_{1/2}^{+}\text{ and }u_j+\xi_j=0\text{ on }B_{1/2}^{'}$$ Now for large enough $j$, we have that   $\n\xi_j\n_{C^1(\overline{B_{1/2}^+})}\leq\varepsilon_0$ which is a contradiction.  This implies the claim of the Lemma. 
\end{proof}
We now discuss the applicability of Lemma \ref{dc0} (and subsequent results in this sections). Let $u\in L^\infty(B_1^+)$ solves $$\pde=-\ div\ \mathbf{h}+f\text{ in }B_1,\ \ \ u=g\text{ on }B_1'$$ with $f\in\li.$ Set for $0<r\leq1$, $$v(x)=\frac{u(rx)-u(0)}{\n u\n_{L^\infty(B_1^+)}}\text{ for }x\in B_1^+.$$
Then $\n v\n_{L^\infty(B_1)}\leq1$ and $v$ solves 
$$-\ \text{div}(\langle B\nabla v,\nabla v\rangle^{\frac{p-2}{2}}B\nabla v)=-\ div\ \mathbf{h_1}+f_1\text{ in }B_1,\ \ \ v=g_1\text{ on }B_1'$$ where $B(x)=A(rx)$, $$\mathbf{h_1}(x)=\frac{r^{p-1}}{\n u\n_{L^\infty(B_1^+)}^{p-1}}\big(\mathbf{h}(rx)-\mathbf{h}(0)\big),\ \ \ f_1(x)=\frac{r^p}{\n u\n_{\li}^{p-1}}f(rx),\ \ \ g_1(x)=\frac{g(rx)-g(0)}{\n u\n_{L^\infty(B_1^+)}}$$ for $x\in B_1$. Note that for $x,y\in B_1$, $$\n[B-I](x)-[B-I](y)\n=\n A(rx)-A(ry)\n\leq r^\gamma\n A-I\n_{C^\gamma(B_1)}\n x-y\n^\gamma,$$
and \begin{align*}
\n f_1\n_{\li}&=\frac{r^p}{\n u\n_{L^\infty(B_1^+)}^{p-1}}\bigg(\int_{B_1^+}|f(rx)|^qdx\bigg)^{1/q}=\frac{r^{p-2/q}}{\n u\n_{L^\infty(B_1^+)}^{p-1}}\bigg(\int_{B_r}|f(y)|^qdy\bigg)^{1/q}\\
&=\frac{r^{p-2/q}}{\n u\n_{L^\infty(B_1^+)}^{p-1}}\n f\n_{L^q(B_r^+)}\leq\frac{r}{\n u\n_{L^\infty(B_1^+)}^{p-1}}\n f\n_{L^q(B_r^+)} .
\end{align*}
Therefore choosing $r\in(0,1)$, small enough, we can ensure $\n B-I\n_{C^\gamma(B_1)}\leq\delta$, $\n \mathbf{h_1}\n_{C^\alpha(\overline{B^+})}\leq\delta$, $\n f_1\n_{L^q(B_1^{+})}\leq \delta$, $\n g_1\n_{C^{1,\gamma}(B_1^{'})}\leq\delta$ so that we can apply Lemma \ref{dc0}. Also for $r\in(0,1)$ small enough, $|\nabla v|\leq1$ in $B_1^+$, using Proposition \ref{ci}. In order to achieve the final result we need to rescale back from $v$ to $u$.

\begin{lemma}\label{dc}
There exists a $\lambda_0\in(0,1/2)$ and $\delta_0>0$ such that if $\n \mathbf{h}\n_{C^\alpha(\overline{B^+})}\leq\delta_0$, $\n f\n_{L^q(B_1^{+})}\leq \delta_0$, $\n g\n_{C^{1,\gamma}(B_1^{'})}\leq\delta_0$ and $\n A-I\n_{C^\gamma(B_1^{+})}\leq\delta_0$ with $A\in\mathcal{M}_\lambda$ and $u\in W^{1,p}(B_1^+)$ is a weak solution to $$\pde=-\ div\ \mathbf{h}+f\text{ in }B_{1}^{+},\ u=g\text{ on }B_{1}',$$ with $\n u\n_{L^\infty(B_1^+)}\leq1$, then
$$\sup_{x\in B_{\lambda_0}^+}|u(x)-[u(0)+\nabla u(0)\cdot x]|\leq\lambda_0^{1+\gamma},\ \ \ \sup_{x\in B_{\lambda_0}^+}|\nabla u(x)-\nabla u(0)|\leq\lambda_0^\gamma.$$
\end{lemma}
\begin{proof}
Let $\varepsilon>0$ which will be fixed later. Then previous Lemma will give a $\delta_0$ so that if $\n \mathbf{h}\n_{C^\alpha(\overline{B^+})}\leq\delta_0$, $\n f\n_{L^\infty(B_1^{+})}\leq\delta_0$, $\n g\n_{C^{1,\gamma}(B_1^{+})}\leq\delta_0$, $\n A-I\n_{C^\gamma(B_1^{+})}\leq\delta_0$ and $u\in W^{1,p}(B_1^{+})$ is a weak solution to $$\pde=-\ div\ \mathbf{h}+f\text{ in }B_{1}^{+},\ u=g\text{ on }B_{1}',$$ with $\n u\n_{L^\infty(B_1^{+})}\leq1$, we have a $\xi\in C^1(\overline{B_{1/2}^+})$, with $\n\xi\n_{C^1(\overline{B_{1/2}^+})}\leq\varepsilon$ such that $u+\xi\in\mathcal{A}(p,d)$ (after extending $\xi$ to full of $B_1^+\cup B_{1}'$ with $\n\xi\n_{C^1(B_{1}^+\cup B_{1}')}\leq\varepsilon$ followed by an odd reflection of $u+\xi$ to the bottom half of $B_1$). Then for $x\in B_{\lambda_0}^+$ with $\lambda_0\in(0,1/2)$, using Theorem \ref{quant}, we have\begin{align*}
|u(x)-[u(0)+\nabla u(0)\cdot x]|\leq&|(u+\xi)(x)-[(u+\xi)(0)+\nabla (u+\xi)(0)\cdot x]|\\
&+|\xi(x)|+|\xi(0)|+|\nabla\xi(0)\cdot x|\\
\leq&C_p(1+\varepsilon)\lambda_0^{p'+\tau_0}+3\varepsilon.
\end{align*} 
Also 
\begin{align*}
|\nabla u(x)-\nabla u(0)|&\leq |\nabla (u+\xi)(x)-\nabla (u+\xi)(0)|+|\nabla\xi(x)|+|\nabla\xi(0)|\\
&\leq C_p(1+\varepsilon)\lambda_0^{p'-1+\tau_0}+2\varepsilon.
\end{align*}
We choose $\lambda_0\in(0,1/2)$ small enough, to ensure $$C_p(1+1)\lambda_0^{p'+\tau_0}\leq\frac{1}{2}\lambda_0^{1+\gamma}$$ and we fix $\varepsilon=\lambda_0^{1+\gamma}/6$. This completes the proof.
\end{proof}
Let us fix the $\lambda_0$ as in Lemma \ref{dc}. Then we have the following corollary as a consequence of Lemma \ref{dc} and triangular inequality.
\begin{corollary}\label{dc5}
There exists a $\delta_0>0$, such that if $\n\mathbf{h}\n_{C^\alpha(\overline{B^+})}\leq\delta_0, \n f\n_{L^q(B_1^{+})}\leq \delta_0, \n g\n_{C^{1,\gamma}(B_1')}\leq\delta_0$ and $\n A-I\n_{C^\gamma(B_1^{+})}\leq\delta_0$ with $A\in\mathcal{M}_\lambda$ and $u\in W^{1,p}(B_1^+)$ is a weak solution to $$\pde=-\ div\ \mathbf{h}+f\text{ in }B_{1}^{+},\ u=g\text{ on }B_{1}',$$ with $\n u\n_{L^\infty(B_1^+)}\leq1$, then
$$\sup_{x\in B_{\lambda_0}^+}|u(x)-u(0)|\leq\lambda_0^{1+\gamma}+|\nabla u(0)|\lambda_0.$$
\end{corollary}

\begin{theorem}\label{dC}
There exist a constant $C$ and $\delta_0>0$ such that if $A\in\mathcal{M}_\lambda$ and $u\in W^{1,p}(B_1^{+})$ solves $$\pde=-\ div\ \mathbf{h}+f\text{ in }B_{1}^{+},\ u=g\text{ on }B_{1}',$$ with $\n u\n_{L^\infty(B_1^+)}\leq1$, $\n \mathbf{h}\n_{C^\Gamma(\overline{B^+})}\leq\delta_0$, $\n f\n_{L^q(B_1^{+})}\leq \delta_0$, $\n g\n_{C^{1,\gamma}(B_1^{'})}\leq\delta_0$ and $\n A-I\n_{C^\gamma(B_1^{+})}\leq\delta_0$, then for all $0<r\leq1$
$$\sup_{x\in B_r^+}|u(x)-u(0)|\leq Cr^{1+\gamma}\big(1+|\nabla u(0)|r^{-\gamma}\big),\ \ \sup_{x\in B_r^+}|\nabla u(x)-\nabla u(0)|\leq Cr^\gamma\big(1+|\nabla u(0)|r^{-\gamma}\big).$$
\end{theorem}
\begin{proof}
Let $\delta=\frac{\delta_0}{2}$, with $\delta_0$ as in Corollary \ref{dc5}. Assume $\n \mathbf{h}\n_{C^\Gamma(\overline{B^+})}\leq\delta$, $\n f\n_{L^\infty(B_1^{+})}\leq\delta$, $\n g\n_{C^{1+\gamma}(B_1^{+})}\leq\delta$, $\n A-I\n_{C^\gamma(B_1^{+})}\leq\delta$. Without loss of generality assume $\delta\leq1$. Set $a_0=1$ and for $k=0,1,2,3,\cdots$ $$a_{k+1}=\lambda_0^{1+\gamma}a_k+\frac{1}{\delta}|\nabla u(0)|\lambda_0^{k+1}.$$
Then we claim that $$\sup_{x\in B_{\lambda_0^k}^+}|u(x)-u(0)|\leq a_k.$$  Our claim is true for $k=1$ by Corollary \ref{dc5}. Assume the claim is true for $k$.
Consider the function $w$ defined by $$w(x)=\frac{u(\lambda_0^kx)-u(0)}{a_k}, \text{ for }x\in B_1^{+}.$$ By induction, $\n w\n_{L^\infty(B_1^{+})}\leq1$ and $$-\ div\ (\langle A_1\nabla w,\nabla w\rangle^{\frac{p-2}{2}}A_1\nabla w)=-\ div\ \mathbf{h_1}+f_1\text{ in }B_{1}^{+},\ w=g_1\text{ on }B_{1}',$$ with $A_1=A(\lambda_0^k\ \cdot)$, $$\mathbf{h_1}=\frac{\lambda_0^{k(p-1)}}{a_k^{p-1}}\big(\mathbf{h}(\lambda_0^k\ \cdot)-\mathbf{h}(0)\big),\ \ f_1=\frac{\lambda_0^{kp}}{a_k^{p-1}}f(\lambda_0^k\ \cdot),\ \ g_1=\frac{g(\lambda_0^k\ \cdot)-g(0)}{a_k}.$$ Note that 
$A_1\in\mathcal{M}_\lambda$, $\n A_1-I\n_{C^\gamma(B_1)}\leq\delta_0$ and \begin{align*}
\n f_1\n_{\li}=\frac{\lambda_0^{k(p-2/q)}}{a_k^{p-1}}\n f\n_{L^q(B_{\lambda_0^k}^+)}\leq\n f\n_{L^q(B_{\lambda_0^k}^+)}\leq\delta_0.
\end{align*}
For $x,y\in B_1^+$, $$|\mathbf{h_1}(x)-\mathbf{h_1}(y)|\leq\frac{\lambda_0^{k(p-1)}}{a_k^{p-1}}\n \mathbf{h}\n_{C^{\Gamma}(B_1^+)}\lambda_0^{k\Gamma}|x-y|^\Gamma\leq\delta_0|x-y|^\Gamma.$$
Also $$|g_1(x')|\leq\frac{|\nabla g(0)|\lambda_0^k+\delta\lambda_0^{(1+\gamma)k}}{\lambda_0^{1+\gamma}a_{k-1}+\frac{1}{\delta}|\nabla u(0)|\lambda_0^k}.$$ If $|\nabla g(0)|\lambda_0^k\leq\delta\lambda_0^{(1+\gamma)k}$ then$$|g_1(x')|\leq\frac{2\delta\lambda_0^{(1+\gamma)k}}{\lambda_0^{(1+\gamma)k}}=2\delta.$$
If $\delta\lambda_0^{(1+\gamma)k}\leq|\nabla g(0)|\lambda_0^k$ then$$|g_1(x')|\leq\frac{2|\nabla g(0)|\lambda_0^k}{\frac{1}{\delta}|\nabla u(0)|\lambda_0^k}=2\delta.$$
Now, for $x',y'\in B_1'$, $$|\nabla g_1(x')-\nabla g_1(y')|\leq\frac{\lambda_0^{(1+\gamma)k}}{a_k}\n g\n |x'-y'|^\gamma\leq\delta|x'-y'|^\gamma.$$
Using Corollary \ref{dc5} $$\sup_{x\in B_{\lambda_0}^+}|w(x)-w(0)|\leq\lambda_0^{1+\gamma}+|\nabla w(0)|\lambda_0$$ which gives$$\sup_{x\in B_{\lambda_0}^+}\bigg|\frac{u(\lambda_0^kx)-u(0)}{a_k}\bigg|\leq\lambda_0^{1+\gamma}+\bigg|\frac{\lambda_0^k}{a_k}\nabla u(0)\bigg|\lambda_0,$$ consequently
$$\sup_{x\in B_{\lambda_0^{k+1}}^+}|u(x)-u(0)|\leq\lambda_0^{1+\gamma}a_k+|\nabla u(0)|\lambda_0^{k+1}\leq\lambda_0^{1+\gamma}a_k+\frac{1}{\delta}|\nabla u(0)|\lambda_0^{k+1}=a_{k+1}.$$
Thus our claim is established.\\
Also $$\sup_{x\in B_{\lambda_0}^+}|\nabla w(x)-\nabla w(0)|\leq\lambda_0^\gamma\Rightarrow\sup_{x\in B_{\lambda_0^{k+1}}^+}|\nabla u(x)-\nabla u(0)|\leq \frac{a_k}{\lambda_0^k}\lambda_0^\gamma.$$
For $0<r\leq\lambda_0$, choose $k$ such that $\lambda_0^{k+1}<r\leq\lambda_0^k$. Set $$b_k=\frac{a_k}{\lambda_0^{k(1+\gamma)}}.$$Note that $$b_{k+1}=\frac{a_{k+1}}{\lambda_0^{(k+1)(1+\gamma)}}=\frac{\lambda_0^{1+\gamma}a_k+\frac{1}{\delta}|\nabla u(0)|\lambda_0^{k+1}}{\lambda_0^{(k+1)(1+\gamma)}}=b_k+\frac{1}{\delta}|\nabla u(0)|\lambda_0^{-(k+1)\gamma}.$$
Then 
\begin{align*}
\sup_{x\in B_r^+}\frac{|u(x)-u(0)|}{r^{1+\gamma}}\leq&\ \sup_{x\in B_{\lambda_0^k}^+}\frac{|u(x)-u(0)|}{\lambda_0^{(k+1)(1+\gamma)}}\leq\frac{b_k}{\lambda_0^{1+\gamma}}\\
\leq&\ \dfrac{b_0+\frac{1}{\delta}|\nabla u(0)|\sum_{i=1}^k\lambda_0^{-\gamma i}}{\lambda_0^{1+\gamma}}\\
=&\ \dfrac{1+\frac{1}{\delta}|\nabla u(0)|\lambda_0^{-\gamma }\frac{\lambda_0^{-\gamma k}-1}{\lambda_0^{-\gamma}-1}}{\lambda_0^{1+\gamma}}, \text{ as }b_0=1\\
\leq&\ \lambda_0^{-1-\gamma}+\frac{1}{\delta}\frac{|\nabla u(0)|}{\lambda_0^{1+\gamma}(1-\lambda_0^\gamma)}\lambda_0^{-\gamma k}\\
\leq&\ \lambda_0^{-1-\gamma}+\frac{1}{\delta}\frac{|\nabla u(0)|}{\lambda_0^{1+\gamma}(1-\lambda_0^\gamma)}r^{-\gamma}.
\end{align*}
For $\lambda_0<r\leq1$ we see as $\n u\n_{L^\infty(B_1^+)}\leq1$, $$\sup_{x\in B_r^+}\frac{|u(x)-u(0)|}{r^{1+\gamma}}\leq\frac{2}{\lambda_0^{1+\gamma}}.$$
On the other hand 
\begin{align*}
\sup_{x\in B_r^+}\frac{|\nabla u(x)-\nabla u(0)|}{r^\gamma}&\leq\sup_{x\in B_{\lambda_0^k}^+}\frac{|\nabla u(x)-\nabla u(0)|}{\lambda_0^{(k+1)\gamma}}\leq\frac{a_{k-1}}{\lambda_0^{k-1}}\lambda_0^\gamma\frac{1}{\lambda_0^{(k+1)\gamma}}=\frac{a_{k-1}}{\lambda_0^{(k-1)(1+\gamma)}}\frac{1}{\lambda_0^\gamma}=\frac{b_{k-1}}{\lambda_0^\gamma}\\
&\leq\frac{1}{\lambda_0^\gamma}+\frac{1}{\delta}|\nabla u(0)|\frac{\lambda_0^{-k\gamma}}{1-\lambda_0^\gamma}\leq \frac{1}{\lambda_0^\gamma}+\frac{1}{\delta}|\nabla u(0)|\frac{r^{-\gamma}}{1-\lambda_0^\gamma}.
\end{align*}  
We choose $C=C(\lambda_0,\gamma)$ appropriately. Rename the $\delta$ to be $\delta_0$.
\end{proof}
\begin{theorem}\label{dA}
Let $A\in\mathcal{M}_\lambda$ and $u\in W^{1,p}(B_1^{+})$ with $\n u\n_{L^\infty(B_1^+)}\leq1$ solves 
\begin{equation}
\pde=-\ div\ \mathbf{h}+f\text{ in }B_{1}^{+},\ u=g\text{ on }B_{1}'.
\end{equation}
Then there exist constants $C$ and $\delta_0>0$ such that if 
\begin{equation}\label{dh}
\n \mathbf{h}\n_{C^\Gamma(\overline{B_1^+})}\leq\delta_0,\ \n f\n_{L^q(B_1^{+})}\leq \delta_0,\ \n g\n_{C^{1,\gamma}(B_1^{'})}\leq\delta_0,\ \n A-I\n_{C^\gamma(B_1^{+})}\leq\delta_0,
\end{equation}
one has for all $0<r\leq1$
\begin{equation*}
\sup_{x\in B_r^+}|u(x)-[u(0)+\nabla u(0)\cdot x]|\leq Cr^{1+\gamma},\ \ \ \sup_{x\in B_r^+}|\nabla u(x)-\nabla u(0)|\leq Cr^\gamma.
\end{equation*}
\end{theorem}
\begin{proof}
If $|\nabla u(0)|\leq r^\gamma$, then we have from Theorem \ref{dC},
\begin{align*}
\sup_{x\in B_r^+}|u(x)-[u(0)+\nabla u(0)\cdot x]|\leq&\ \sup_{x\in B_r^+}|u(x)-u(0)|+|\nabla u(0)|r\\
\leq&\ (C+1)r^{1+\gamma}.
\end{align*}
Also similarly $$\sup_{x\in B_r^+}|\nabla u(x)-\nabla u(0)|\leq 2Cr^\gamma.$$
So let us assume $|\nabla u(0)|>r^\gamma$ and define $\mu=|\nabla u(0)|^{1/\gamma}$. Set for $x\in B_1^+$ $$v(x)=\frac{u(\mu x)-u(0)}{\mu^{1+\gamma}}.$$(Note that in order to define $v$ we need to have $\mu\leq1$ which can be assumed without loss generality as discussed before Lemma \ref{dc}.)  
From Theorem \ref{dC} one has $$\sup_{x\in B_1^+}|v(x)|=\sup_{x\in B_\mu^+}\frac{|u(x)-u(0)|}{\mu^{1+\gamma}}\leq C.$$
Now $v(0)=0$, $|\nabla v(0)|=1$ and $v$ satisfies
$$-\ div\ (\langle A_1\nabla v,\nabla v\rangle^{\frac{p-2}{2}}A_1\nabla v)=-\ div\ \mathbf{h_1}+f_1\text{ in }B_{1}^{+},\ v=g_1\text{ on }B_{1}',$$ with $A_1\in \mathcal{M}_\lambda$, $$\mathbf{h_1}=\frac{\mu^{p-1}}{\mu^{(1+\gamma)(p-1)}}\big(\mathbf{h}(\mu\ \cdot)-\mathbf{h}(0)\big),\ \ f_1=\frac{\mu^{p}}{\mu^{(1+\gamma)(p-1)}}f(\mu\ \cdot),\ \ g_1=\frac{g(\mu\ \cdot)-g(0)}{\mu^{1+\gamma}}.$$ Note that 
 $\n A_1-I\n_{C^\gamma(B_1)}\leq\delta_0$, $\n \mathbf{h_1}\n_{C^\Gamma(\overline{B_1^+})}\leq\delta_0$, $\n f_1\n_{L^q(B_1^{+})}\leq \delta_0$ and $\n g_1\n_{C^{1,\gamma}(B_1^{'})}\leq2\delta_0$. 
Applying $C^{1,\beta}$ estimates, see Proposition \ref{ci}, we have a $\rho_0$ independent of $w$ such that $$|\nabla v(x)|>\frac{1}{2}\text{ for all }x\in B_{\rho_0}^+.$$
At this point we apply the result from quadratic theory i.e. $p=2$ case, see Theorem \ref{d0} in section \ref{sd}. Consider the PDE 
$$-\ div\ \mathbf{A}(x,\nabla v)=-\ div\ \mathbf{h_1}+f_1\text{ in }B_{\rho_0}^{+},\ v=g_1\text{ on }B_{\rho_0}',$$
with $$\mathbf{A}(x,\zeta)=\langle A(x)\zeta,\zeta\rangle^{\frac{p-2}{2}}A(x)\zeta\text{ for }x\in B_{\rho_0}^+,\ \frac{1}{2}\le|\zeta|\leq1$$ but $\mathbf{A}$ is continuously differentiable in $\zeta$ variable, $\partial_{\zeta_j}\mathbf{A}^i(x,\zeta)\xi_i\xi_j\geq\lambda'|\xi|^2$, $|\partial_{\zeta_j}\mathbf{A}^i(x,\zeta)|\leq L$, $|\mathbf{A}^i(x,\zeta)-\mathbf{A}^i(y,\zeta)|\leq L(1+|\zeta|)|x-y|^\gamma$. Then there exists a $C=C(\rho_0,\gamma)\geq0$ independent of $u$ so that $$\sup_{x\in B_r^+}|v(x)-[v(0)+\nabla v(0)\cdot x]|\leq Cr^{1+\gamma}$$ i.e. $$\sup_{x\in B_r^+}|u(\mu x)-u(0)-\nabla u(0)\cdot \mu x|\leq C(\mu r)^{1+\gamma}$$for all $0<r\leq\rho_0/2$. In other words $$\sup_{x\in B_r^+}|u(x)-u(0)-\nabla u(0)\cdot x|\leq Cr^{1+\gamma}$$ for all $0<r\leq\mu\rho_0/2$. Similarly $$\sup_{x\in B_r^+}|\nabla u(x)-\nabla u(0)|\leq Cr^\gamma$$ for all $0<r\leq\mu\rho_0/2$.
For $\mu\rho_0/2\leq r<\mu$, 
\begin{align*}
\sup_{x\in B_r^+}|u(x)-u(0)-\nabla u(0)\cdot x|&\leq \sup_{x\in B_\mu^+}|u(x)-u(0)-\nabla u(0)\cdot x|\\
&\leq C\mu^{1+\gamma}(1+|\nabla u(0)|\mu^{-\gamma})+|\nabla u(0)|\mu\\
&=C\mu^{1+\gamma}+(C+1)\mu^{1+\gamma}=(2C+1)\mu^{1+\gamma}\\
&\leq (2C+1)(2/\rho_0)^{1+\gamma}r^{1+\gamma}
\end{align*} and also similarly $$\sup_{x\in B_r^+}|\nabla u(x)-\nabla u(0)|\leq 2C(2/\rho_0)^\gamma r^\gamma.$$ This completes the proof.
\end{proof}
Note that we have the following result for interior case which can be proved more easily than Theorem \ref{dA} due to absence of Dirichlet data. We mention that, this result (conclusion \eqref{d02}) in the case of $\mathbf{h}=0$ was proved by Araujo and Zhang in \cite{az}.
\begin{theorem}\label{dB}
There exist a constant $C$ and $\delta_0>0$ such that if $A\in\mathcal{O}_\lambda$ (see \eqref{0}) and $u\in W^{1,p}(B_1)$ solves $$\pde=-\ div\ \mathbf{h}+f\text{ in }B_{1},$$ with $\n u\n_{L^\infty(B_1)}\leq1$, $\n \mathbf{h}\n_{C^\Gamma(B_1)}\leq\delta_0$, $\n f\n_{L^q(B_1)}\leq \delta_0$ and $\n A-I\n_{C^\gamma(B_1)}\leq\delta_0$, then for all $0<r\leq1$,
\begin{equation}\label{d02}
\sup_{x\in B_r}|u(x)-[u(0)+\nabla u(0)\cdot x]|\leq Cr^{1+\gamma},
\end{equation}$$\sup_{x\in B_r^+}|\nabla u(x)-\nabla u(0)|\leq Cr^\gamma.$$
\end{theorem}
To complete the proof of Theorem \ref{main1}, it is enough to proof the following: Under the assumptions of Theorem \ref{dA} there exist a constant $C$ and $\delta_0>0$ such that if the conditions in \eqref{dh} is satisfied, then for $x,y\in B_{1/2}^+$ 
\begin{equation}\label{d00}
|u(y)-[u(x)+\nabla u(x)\cdot (y-x)]|\leq C|y-x|^{1+\gamma},\ \ \ |\nabla u(x)-\nabla u(y)|\leq C|x-y|^\gamma.
\end{equation}
So fix $x,y\in B_{1/2}^+$. Let us denote $(x_1,0)$ by $\overline{x}$. By Theorem \ref{dA} we have \begin{equation}\label{dL}
|u(y)-u(\overline{x})-\nabla u(\overline{x})\cdot(y-\overline{x})|\leq C|y-\overline{x}|^{1+\gamma}.
\end{equation}
\textit{\textbf{Case Ia:}} $y\in B_{x_2/2}(x)$, $|\nabla u(\overline{x})|\leq(2x_2/\rho_0)^\gamma$, where $\rho_0>0$ to be fixed later.\\
Then we have from \eqref{dL}, using triangular inequality $$|u(y)-u(\overline{x})|\leq Cx_2^{1+\gamma}.$$
Set $v(z)=\big(u(x+x_2z)-u(\overline{x})\big)/x_2^{1+\gamma}$ for $z\in B_1$. The $v$ satisfies $$-\ div\ (\langle A_1\nabla v,\nabla v\rangle^{\frac{p-2}{2}}A_1\nabla v)=-\ div\ \mathbf{h_1}+f_1\text{ in }B_{1}$$ with $A_1=A(x+x_2\ \cdot)\in \mathcal{O}_\lambda$, $$\mathbf{h_1}=\frac{x_2^{p-1}}{x_2^{(1+\gamma)(p-1)}}\big(\mathbf{h}(x+x_2\ \cdot)-\mathbf{h}(x)\big),\ \ f_1=\frac{x_2^{p}}{x_2^{(1+\gamma)(p-1)}}f(x+x_2\ \cdot).$$
Hence using Theorem \ref{dB} (the smallness of data required there follows), for $|z|\leq1$, $|v(z)-v(0)-\nabla v(0)\cdot z|\leq C|z|^{1+\gamma}$ and $|\nabla v(z)-\nabla v(0)|\leq C|z|^\gamma$. Consequently for $y\in B_{x_2/2}(x)$, $$|u(y)-u(x)-\nabla u(x)\cdot(y-x)|\leq C|y-x|^{1+\gamma},\ \ \ |\nabla u(y)-\nabla u(x)|\leq C|y-x|^\gamma.$$
\textbf{\textit{Case Ib:}} $y\in B_{x_2/2}(x)$, $|\nabla u(\overline{x})|\geq(2x_2/\rho_0)^\gamma$.\\
Set $\mu=|\nabla u(\overline{x})|^{1/\gamma}$ and assume $\mu\leq1$ without loss of generality as discussed in the proof of Theorem \ref{dA}. For $z\in B_1^+\cup B_1'$ define $w(z)=\big(u(\overline{x}+\mu z)-u(\overline{x})\big)/\mu^{1+\gamma}$. From Theorem \ref{dC} one has $$\sup_{z\in B_1^+}|w(z)|=\sup_{z\in B_\mu^+}\frac{|u(\overline{x}+z)-u(\overline{x})|}{\mu^{1+\gamma}}\leq C.$$ Now $w(0)=0$, $|\nabla w(0)|=1$ and $v$ satisfies
$$-\ div\ (\langle A_2\nabla w,\nabla w\rangle^{\frac{p-2}{2}}A_2\nabla w)=-\ div\ \mathbf{h_2}+f_2\text{ in }B_{1}^{+},\ v=g_2\text{ on }B_{1}',$$ with $A_2\in \mathcal{M}_\lambda$, $$\mathbf{h_2}=\frac{\mu^{p-1}}{\mu^{(1+\gamma)(p-1)}}\big(\mathbf{h}(\overline{x}+\mu\ \cdot)-\mathbf{h}(\overline{x})\big),\ \ f_2=\frac{\mu^{p}}{\mu^{(1+\gamma)(p-1)}}f(\overline{x}+\mu\ \cdot),\ \ g_2=\frac{g(\overline{x}+\mu\ \cdot)-g(\overline{x})}{\mu^{1+\gamma}}.$$ 
Note that $\n \mathbf{h_2}\n_{C^\Gamma(\overline{B_1^+})}\leq\delta_0,\ \n f_2\n_{L^q(B_1^{+})}\leq \delta_0,\ \n g_2\n_{C^{1,\gamma}(B_1^{'})}\leq2\delta_0,\ \n A_2-I\n_{C^\gamma(B_1^{+})}\leq\delta_0$.
By $C^{1,\beta}$ estimate i.e Proposition \ref{ci}, we have again a universal $\rho_0$ such that $|\nabla w|\geq1/2$ in $B_{\rho_0}^+$. Now going to $B_1^+$ (and applying Theorem \ref{d0} there, as in the proof of Theorem \ref{dA}), and then rescaling back to $B_{\rho_0}^+$ one has $w\in C^{1,\gamma}(\overline{B_{\rho_0}^+})$ with $\n w\n_{C^{1,\gamma}(\overline{B_{\rho_0}^+})}$ depending on $\rho_0,\gamma$. Note that $x_2/\mu\leq\rho_0/2$. So for $|z|\leq\rho_0/4$, $|w((0,x_2/\mu)+z)-w(0,x_2/\mu)-\nabla w(0,x_2/\mu)\cdot z|\leq C(\rho_0,\gamma)|z|^{1+\gamma}$, $|\nabla w((0,x_2/\mu)+z)-\nabla w(0,x_2/\mu)|\leq C(\rho_0,\gamma)|z|^\gamma$ and consequently for $|y-x|<x_2/2\leq\mu\rho_0/4$,$$|u(y)-u(x)-\nabla u(x)\cdot (y-x)|\leq C|y-x|^{1+\gamma},\ \ \ |\nabla u(y)-\nabla u(x)|\leq C|y-x|^\gamma.$$
\textbf{\textit{Case II:}} $|y-x|\geq x_2/2$.\\
Note that using \eqref{dL} and Theorem \ref{dA},
\begin{align*}
|u(y)-u(x)-\nabla u(x)\cdot (y-x)|&\leq |u(y)-u(\overline{x})-\nabla u(\overline{x})\cdot(y-\overline{x})|+|\nabla u(x)-\nabla u(\overline{x})||y-x|\\
&\ \ \ +|u(x)-u(\overline{x})-\nabla u(\overline{x})\cdot(x-\overline{x})|\\
&\leq C(|y-\overline{x}|^{1+\gamma}+|x-\overline{x}|^{1+\gamma}+x_2^\gamma|y-x|)\\
&\leq C_1|y-x|^{1+\gamma}.
\end{align*}
Also \begin{align}
|\nabla u(y)-\nabla u(x)|&\leq|\nabla u(y)-\nabla u(\overline{x})|+|\nabla u(x)-\nabla u(\overline{x})|\\
&\leq C|y-\overline{x}|^\gamma+Cx_2^\gamma\leq C_1|y-x|^\gamma.\label{d01}
\end{align}
This completes the proof Theorem \ref{main1}.

\subsection{Proof of Theorem \ref{main2}}
As in the case of Dirichlet data, with out loss of generality, we assume the situation of \eqref{E2} (where Dirichlet data replaced with conormal one) with $A(0)=I$ and $s=1$ . Let $\alpha\geq\gamma$.
\begin{lemma}\label{cc0}
Let $u\in W^{1,p}(B_1^{+})$ be a weak solution of
\[   \left\{
\begin{array}{ll} 
      \pde=-\ div\ \mathbf{h}+f\text{ in }B_{1}^{+}, \\
      \ \ \ \langle A\nabla u,\nabla u\rangle^{\frac{p-2}{2}}\langle A\nabla u,\nu\rangle=|g|^{p-2}g\text{ on }B_1',\\
   
\end{array} 
\right. \]
with $\n u\n_{L^\infty(B_{1}^{+})}\leq1$, $A\in\mathcal{M}_\lambda$. Then for given $\varepsilon>0$, there exists $\delta=\delta(p,\varepsilon,\lambda)>0$ such that if $\n \mathbf{h}\n_{C^\alpha(\overline{B^+})}\leq\delta$, $\n f\n_{L^q(B_1^{+})}\leq \delta$, $\n g\n_{C^{\gamma}(B_1^{'})}\leq\delta$ and $\n A-I\n_{C^\gamma(B_1^{+})}\leq\delta$, then we can find a corrector $\xi\in C^1(\overline{B_{1/2}^+})$, with $$|\xi(x)|\leq\varepsilon\text{ and }|\nabla\xi(x)|\leq\varepsilon\text{ for all }x\in B_{1/2}^+$$ satisfying$$-\ \Delta_p(u+\xi)=0\text{ in }B_{1/2}^{+} \text{ and }\nabla (u+\xi)\cdot\nu=0\text{ on }B_{1/2}'.$$
\end{lemma}
\begin{proof}
The proof is similar as in the case of Dirichlet data, see Lemma \ref{dc0}. We have similar $C^{1,\beta}$ estimates upto boundary, see Proposition \ref{cj}. 
\end{proof}

\begin{lemma}\label{cc}
There exists a $\lambda_0\in(0,1/2)$ and $\delta_0>0$, such that if $A\in\mathcal{M}_\lambda$, $\n \mathbf{h}\n_{C^\alpha(\overline{B^+})}\leq\delta_0$, $\n f\n_{L^q(B_1^{+})}\leq \delta_0$, $\n g\n_{C^{\gamma}(B_1^{'})}\leq\delta_0$ and $\n A-I\n_{C^\gamma(B_1^{+})}\leq\delta_0$ and $u\in W^{1,p}(B_1)$ is a weak solution to \[   \left\{
\begin{array}{ll} 
      \pde=-\ div\ \mathbf{h}+f  \text{ in }B_1^{+}, \\
      \ \ \ \langle A\nabla u,\nabla u\rangle^{\frac{p-2}{2}}\langle A\nabla u,\nu\rangle=|g|^{p-2}g\text{ on }B_1',\\
   
\end{array} 
\right. \]
with $\n u\n_{L^\infty(B_1^{+})}\leq1$, then
$$\sup_{x\in B_{\lambda_0}^+}|u(x)-[u(0)+\nabla u(0)\cdot x]|\leq\lambda_0^{1+\gamma},\ \ \ \sup_{x\in B_{\lambda_0}^+}|\nabla u(x)-\nabla u(0)|\leq\lambda_0^\gamma.$$
\end{lemma}
\begin{proof}
Here we have to do an odd reflection instead of even one in the Dirichlet case. Rest of the proof is similar to the proof of Lemma \ref{dc}.
\end{proof}

As before let us fix the $\lambda_0$ as in Lemma \ref{cc} and get the corollary below.
\begin{corollary}\label{cc5}
There exists a $\delta_0>0$, such that if $A\in\mathcal{M}_\lambda$, $\n \mathbf{h}\n_{C^\alpha(\overline{B^+})}\leq\delta_0$, $\n f\n_{L^q(B_1^{+})}\leq \delta_0$, $\n g\n_{C^{\gamma}(B_1^{'})}\leq\delta_0$ and $\n A-I\n_{C^\gamma(B_1^{+})}\leq\delta_0$ and $u\in W^{1,p}(B_1)$ is a weak solution to \[   \left\{
\begin{array}{ll} 
      \pde=-\ div\ \mathbf{h}+f  \text{ in }B_1^{+}, \\
      \ \ \ \langle A\nabla u,\nabla u\rangle^{\frac{p-2}{2}}\langle A\nabla u,\nu\rangle=|g|^{p-2}g\text{ on }B_1',\\
   
\end{array} 
\right. \]
 with $\n u\n_{L^\infty(B_1^{+})}\leq1$, then
$$\sup_{x\in B_{\lambda_0}^+}|u(x)-u(0)|\leq\lambda_0^{1+\gamma}+|\nabla u(0)|\lambda_0.$$
\end{corollary}
\begin{theorem}\label{nC}
There exist constants $C$ and $\delta_0>0$ such that if $A\in\mathcal{M}_\lambda$ and $u\in W^{1,p}(B_1^{+})$ solves\[   \left\{
\begin{array}{ll} 
      \pde=-\ div\ \mathbf{h}+f  \text{ in }B_1^{+}, \\
      \ \ \ \langle A\nabla u,\nabla u\rangle^{\frac{p-2}{2}}\langle A\nabla u,\nu\rangle=|g|^{p-2}g\text{ on }B_1',\\
   
\end{array} 
\right. \] with $\n u\n_{L^\infty(B_1^{+})}\leq1$, $\n \mathbf{h}\n_{C^\Gamma(\overline{B^+})}\leq\delta_0$, $\n f\n_{L^q(B_1^{+})}\leq \delta_0$, $\n g\n_{C^{\gamma}(B_1^{'})}\leq\delta_0$ and $\n A-I\n_{C^\gamma(B_1^{+})}\leq\delta_0$, then for all $0<r\leq1$,
$$\sup_{x\in B_r^+}|u(x)-u(0)|\leq Cr^{1+\gamma}\big(1+|\nabla u(0)|r^{-\gamma}\big),\ \ \sup_{x\in B_r^+}|\nabla u(x)-\nabla u(0)|\leq Cr^\gamma\big(1+|\nabla u(0)|r^{-\gamma}\big).$$
\end{theorem}
\begin{proof}
Note that $\exists$ a $K>0$ such that $|g(0)|\leq K|\nabla u(0)|$. Let $\delta=\frac{\delta_0}{1+K}$. Assume $\n f\n_{L^\infty(B_1^{+})}\leq\delta$, $\n g\n\leq\delta$, $\n A-I\n_{C^\gamma(B_1^{+})}\leq\delta$. Without loss of generality assume $\delta\leq1$. Set $a_0=1$ and for $k=0,1,2,3,\cdots$ $$a_{k+1}=\lambda_0^{1+\gamma}a_k+\frac{1}{\delta}|\nabla u(0)|\lambda_0^{k+1}.$$
Then we claim that $$\sup_{x\in B_{\lambda_0^k}^+}|u(x)-u(0)|\leq a_k.$$  Our claim is true for $k=1$, by Corollary \ref{cc5}. Assume the claim is true for $k$.
Consider the function $w$ defined by $$w(x)=\frac{u(\lambda_0^kx)-u(0)}{a_k}, \text{ for }x\in B_1^{+}.$$ By induction, $\n w\n_{L^\infty(B_1^{+})}\leq1$ and \[   \left\{
\begin{array}{ll} 
      -\ div\ (\langle A_1\nabla w,\nabla w\rangle^{\frac{p-2}{2}}A_1\nabla w)=-\ div\ \mathbf{h_1}+f_1  \text{ in }B_1^{+}, \\
      \ \ \ \ \ \langle A_1\nabla w,\nabla w\rangle^{\frac{p-2}{2}}\langle A_1\nabla w,\nu\rangle=|g_1|^{p-2}g_1\text{ on }B_1\cap\partial\rn_+,\\
   
\end{array} 
\right. \]  with $A_1=A(\lambda_0^k\ \cdot)$, $$\mathbf{h_1}=\frac{\lambda_0^{k(p-1)}}{a_k^{p-1}}\big(\mathbf{h}(\lambda_0^k\ \cdot)-\mathbf{h}(0)\big),\ \ f_1=\frac{\lambda_0^{kp}}{a_k^{p-1}}f(\lambda_0^k\ \cdot),\ \ g_1=\frac{\lambda_0^k}{a_k}g(\lambda_0^k\ \cdot).$$ Note that 
$A_1\in\mathcal{M}_\lambda$, $\n A_1-I\n_{C^\gamma(B_1)}\leq\delta_0$ and as before $\n \mathbf{h_1}\n_{C^\Gamma(\overline{B^+})}\leq\delta_0$, $\n f_1\n_{\li}\leq\delta_0$ and for $x',y'\in B_1'$
$$|g_1(x')|\leq\frac{\lambda_0^k}{a_k}\big[|g(0)|+\delta\lambda_0^{k\gamma}\big]\leq\frac{K|\nabla u(0)|\lambda_0^k+\delta\lambda_0^{(1+\gamma)k}}{\lambda_0^{1+\gamma}a_{k-1}+\frac{1}{\delta}|\nabla u(0)|\lambda_0^k}\leq (1+K)\delta=\delta_0,$$ 
as well as $$|g_1(x')-g_1(y')|\leq\frac{\lambda_0^{(1+\gamma)k}}{a_k}\n g\n_{C^\gamma(B_1')} |x'-y'|^\gamma\leq\delta|x'-y'|^\gamma.$$
Using Corollary \ref{cc5} as before in Dirichlet case we have
$$\sup_{x\in B_{\lambda_0^{k+1}}^+}|u(x)-u(0)|\leq\lambda_0^{1+\gamma}a_k+|\nabla u(0)|\lambda_0^{k+1}\leq\lambda_0^{1+\gamma}a_k+\frac{1}{\delta}|\nabla u(0)|\lambda_0^{k+1}=a_{k+1}.$$
Thus our claim is established.\\
On the other hand $$\sup_{x\in B_{\lambda_0}^+}|\nabla w(x)-\nabla w(0)|\leq\lambda_0^\gamma\Rightarrow\sup_{x\in B_{\lambda_0^{k+1}}^+}|\nabla u(x)-\nabla u(0)|\leq \frac{a_k}{\lambda_0^k}\lambda_0^\gamma.$$
Now we do exactly same analysis as in Dirichlet case to achieve the result.
\end{proof}
Now we have the analogous result of Theorem \ref{dA} whose proof is quite similar that of Theorem \ref{dA}. Nonetheless we give a short proof of it.
\begin{theorem}
There exist constants $C$ and $\delta_0>0$ such that if $A\in\mathcal{M}_\lambda$ and $u\in W^{1,p}(B_1^{+})$ solves\[   \left\{
\begin{array}{ll} 
      \pde=-\ div\ \mathbf{h}+f  \text{ in }B_1^{+}, \\
      \ \ \langle A\nabla u,\nabla u\rangle^{\frac{p-2}{2}}\langle A\nabla u,\nu\rangle=|g|^{p-2}g\text{ on }B_1',\\
   
\end{array} 
\right. \] with $\n u\n_{L^\infty(B_1^{+})}\leq1$, $\n \mathbf{h}\n_{C^\Gamma(\overline{B^+})}\leq\delta_0$, $\n f\n_{L^q(B_1^{+})}\leq \delta_0$, $\n g\n_{C^{\gamma}(B_1^{'})}\leq\delta_0$ and $\n A-I\n_{C^\gamma(B_1^{+})}\leq\delta_0$, then for all $0<r\leq1$,
$$\sup_{x\in B_r^+}|u(x)-[u(0)+\nabla u(0)\cdot x]|\leq Cr^{1+\gamma},\ \ \ \sup_{x\in B_r^+}|\nabla u(x)-\nabla u(0)|\leq Cr^\gamma.$$
\end{theorem}
\begin{proof}If $|\nabla u(0)|\leq r^\gamma$, then as in the Dirichlet case we are done.
So let us assume $|\nabla u(0)|>r^\gamma$ and define $\mu=|\nabla u(0)|^{1/\gamma}$. Set for $x\in B_1^{+}$ $$v(x)=\frac{u(\mu x)-u(0)}{\mu^{1+\gamma}}.$$
Then Theorem \ref{nC} gives $\n v\n_{L^\infty(B_1^+)}\leq C$.
Now $v(0)=0$, $|\nabla v(0)|=1$ and $v$ satisfies\[   \left\{
\begin{array}{ll} 
      -\ div\ (\langle A_1\nabla v,\nabla v\rangle^{\frac{p-2}{2}}A_1\nabla v)=-\ div\ \mathbf{h_1}+f_1  \text{ in }B_1^+, \\
      \ \ \ \langle A_1\nabla v,\nabla v\rangle^{\frac{p-2}{2}}\langle A_1\nabla v,\nu\rangle=|g_1|^{p-2}g_1\text{ on }B_1',\\
   
\end{array} 
\right. \]
with $A_1\in \mathcal{M}_\lambda$, $\n \mathbf{h_1}\n_{C^\Gamma(\overline{B_1^+})}\leq\delta_0$, $\n f_1\n_{\li}\leq\delta_0$, $\n g_1\n_{C^\gamma(B_1\cap\partial\rn_+)}\leq(1+K)\delta_0$ ($K$ as in last Theorem), $\n A_1-I\n_{C^\gamma(B_1)}\leq\delta_0$. 
Applying local $C^{1,\alpha}$ estimates for genaralized $p-$Poisson equation (Proposition \ref{cj}) we have a $\rho_0$ independent of $w$ such that $$|\nabla v(x)|>\frac{1}{2}\text{ for all }x\in B_{\rho_0}^+.$$
At this point we apply Theorem \ref{C}  in section \ref{sc}. Proceeding as in the Dirichet case we get our result.
\end{proof}
Similar analysis starting from Theorem \ref{dB} upto inequality \eqref{d01} completes the proof of Theorem \ref{main2}.

\section{Quadratic Dirichlet Case}\label{sd}
In this section as well as the next section,  we establish auxiliary  regularity results for equations with quadratic nonlinearities ( both with Dirichlet and Neumann conditions)  that were used in the proofs of our main results by adapting  the ideas of Giaquinta and Giusti (\cite{gg})  in our framework. The reader should note that we have  an extra divergence term and $f\in L^q$ with $q>n$ instead of $L^\infty$. In our proofs, we point out the appropriate  modifications that are needed. Let $\Gamma=1-n/q$.
\begin{theorem}\label{d0}
Let $u\in W^{1,2}(B_1^+)\cup L^\infty(B_1^+)$ solves
\begin{equation}\label{d}
-div\big(\mathbf{A}(\cdot,\nabla u)\big)=-div\ \mathbf{h}+f\ \text{in }B_1^+\ \ \ u=\phi\text{ on }B_1'.
\end{equation}
with $|\partial_{\zeta_j}\mathbf{A}^i|\leq L$, $\partial_{\zeta_j}\mathbf{A}^i\xi_i\xi_j\geq\Lambda|\xi|^2$, $|\mathbf{A}^i(x,\zeta)-\mathbf{A}^i(y,\zeta)|\leq L(1+|\zeta|)|x-y|^\alpha$, $\n f\n_{\li}\leq L$, $\n \mathbf{h}\n_{C^\Gamma(B_1)}\leq L$, $\n\phi\n_{C^{1,\alpha}(B_1')}\leq L$. Also assume $\alpha\leq\Gamma$. Then we have a $C=C(L,\Lambda,..)\geq0$ such that $\n u\n_{C^{1,\alpha}(\overline{B_{3/4}^+})}\leq C$.
\end{theorem}
\begin{proof}
Let $B^+=\{x\in\rn:|x|\leq1,\ x_n>0\}$, $P=\{x\in\rn:|x|<1,\ x_n=0\}$. For $x_0\in\overline{B^+}$, set $$B_r^+=B_r^+(x_0)=\{x\in\rn:|x-x_0|<r,\ x_n>0\}.$$
For $x_0\in P$ and $0<R<dist(x_0,\partial B)$, consider the problem
\begin{equation}\label{s}
-div\big(\mathbf{A}(x_0,Dv)\big)=0\ \text{in }B_R^+(x_0),\ \ \ v=u\text{ on }\partial B_R^+(x_0).
\end{equation}
Note that $v\in W^{2,2}(B_\rho^+)$ for any $\rho<R$. For $k=1,2,\cdots,n-1$, the function $w=D_kv$ solves \begin{equation}\label{de}
-div\big(\mathbf{A}_{p_i}(\cdot,Dv)D_iw\big)=0\ \text{in }B_R^+(x_0)
\end{equation}with $w=0$ on $P_R(x_0):=\partial B_R^+(x_0)\cap P$. As $Dv$ is Holder continuous, by Dirichlet counterpart of Lemma \ref{Cc} one has
\begin{equation}
\int_{B_\rho^+(x_0)}|Dw|^2\leq C\bigg(\frac{\rho}{R}\bigg)^{n-2+2\delta}\int_{B_{R/2}^+(x_0)}|Dw|^2
\end{equation}for some $\delta>0$. Using the equation for $v$, we have that \begin{equation}
\int_{B_\rho^+(x_0)}|D^2v|^2\leq C\bigg(\frac{\rho}{R}\bigg)^{n-2+2\delta}\sum_{k=1}^{n-1}\int_{B_{R/2}^+(x_0)}|DD_kv|^2.
\end{equation}
Applying Caccioppoli's inequality, we conclude that \begin{equation}
\int_{B_\rho^+(x_0)}|D^2v|^2\leq CR^{-2}\bigg(\frac{\rho}{R}\bigg)^{n-2+2\delta}\sum_{k=1}^{n-1}\int_{B_{R}^+(x_0)}|D_kv|^2,
\end{equation} consequently by Poincare's inequality\begin{equation}
\sum_{k=1}^{n-1}\int_{B_\rho^+(x_0)}|D_kv|^2+\int_{B_\rho^+(x_0)}|D_nv-(D_nv)_\rho|^2\leq C\bigg(\frac{\rho}{R}\bigg)^{n+2\delta}\sum_{k=1}^{n-1}\int_{B_{R}^+(x_0)}|D_kv|^2.
\end{equation}
Set $$\Phi^+(x_0,\rho)=\sum_{k=1}^{n-1}\int_{B_\rho^+(x_0)}|D_ku|^2+\int_{B_\rho^+(x_0)}|D_nu-(D_nu)_\rho|^2.$$
Then using triangular inequality, we have \begin{equation}\label{it}
\Phi^+(x_0,\rho)\leq C\bigg(\frac{\rho}{R}\bigg)^{n+2\delta}\Phi^+(x_0,R)+C\int_{B_{R}^+(x_0)}|D(u-v)|^2.
\end{equation}
Using \eqref{d}, \eqref{s} we have that for $\varphi\in H_0^1(B_R^+(x_0))$
\begin{align*}
\int_{B_{R}^+(x_0)}[\mathbf{A}^j(x_0,Du)-\mathbf{A}^j(x_0,Dv)]\partial_j\varphi=&\int_{B_{R}^+(x_0)}[\mathbf{A}^j(x_0,Du)-\mathbf{A}^j(x,Du)]\partial_j\varphi\\
&+\int_{B_{R}^+(x_0)}\mathbf{h}\cdot D\varphi+\int_{B_{R}^+(x_0)}f\varphi.
\end{align*}
Take $\varphi=u-v$ and use ellipticity condition, Holder continuity of $A(\cdot,x),\mathbf{h}$, (here we are absorbing the divergence term) \begin{equation}\label{itt}
\int_{B_{R}^+(x_0)}|D(u-v)|^2\leq CR^\alpha\int_{B_{R}^+(x_0)}(1+|Du|)|D(u-v)|+C\int_{B_{R}^+(x_0)}|f||u-v|.
\end{equation}
Note that $$CR^\alpha\int_{B_{R}^+(x_0)}(1+|Du|)|D(u-v)|\leq CR^\alpha\bigg(\epsilon\int_{B_{R}^+(x_0)}|D(u-v)|^2+\frac{1}{4\epsilon}\int_{B_{R}^+(x_0)}(1+|Du|)^2\bigg)$$and 
\begin{align*}
\int_{B_{R}^+(x_0)}|f||u-v|&\leq\n f\n_{L^{2n/(n+2)}(B_{R}^+(x_0))}\n u-v\n_{L^{2^*}(B_{R}^+(x_0))}\\
&\leq C(n)\n f\n_{L^{2n/(n+2)}(B_{R}^+(x_0))}\n D(u-v)\n_{L^{2}(B_{R}^+(x_0))}\\
&\leq C(n)R^{1+n(1/2-1/q)}\n f\n_{L^q(B_{R}^+(x_0))}\n D(u-v)\n_{L^{2}(B_{R}^+(x_0))}\\
&\leq \varepsilon\n D(u-v)\n_{L^{2}(B_{R}^+(x_0))}^2+\frac{1}{4\varepsilon}C(n)^2R^{n+2(1-n/q)}\n f\n_{L^q(B_{R}^+(x_0))}^2.
\end{align*} Note that in the above we also needed a modification as $f\in L^q$, this was also done in \cite{az}. Using appropriate $\epsilon>0,\varepsilon>0$ we conclude\begin{equation}\label{dd}
\int_{B_{R}^+(x_0)}|D(u-v)|^2\leq CR^{2\alpha}\int_{B_{R}^+(x_0)}(1+|Du|^2)+CR^{n+2(1-n/q)}\n f\n_{L^q(B_{R}^+(x_0))}^2.
\end{equation} Hence from \eqref{itt} we conclude,\begin{equation}\label{iter}
\Phi^+(x_0,\rho)\leq C\bigg(\frac{\rho}{R}\bigg)^{n+2\delta}\Phi^+(x_0,R)+CR^{n+2\Gamma}\n f\n_{L^q(B_{R}^+(x_0))}^2+CR^{2\alpha}\int_{B_{R}^+(x_0)}(1+|Du|^2).
\end{equation}
Similarly one has to get for $B_R(x_0)\subset B^+$ \begin{equation}\label{iter1}
\Phi(x_0,\rho)\leq C\bigg(\frac{\rho}{R}\bigg)^{n+2\delta}\Phi(x_0,R)+CR^{n+2\Gamma}\n f\n_{L^q(B_{R}(x_0))}^2+CR^{2\alpha}\int_{B_{R}(x_0)}(1+|Du|^2)
\end{equation}with $$\Phi(x_0,\rho)=\sum_{k=1}^{n}\int_{B_\rho(x_0)}|D_ku-(D_ku)_\rho|^2.$$

\begin{lemma}
Suppose $u$ as in Theorem \ref{d0} and for some $C\geq0$ and natural number $k\geq1$, for each $x_0\in B_{3/4}^+$, for each $0<R<1/4$ we have \begin{equation}
\int_{B_{R}^+(x_0)}(1+|Du|^2)\leq CR^{(k-1)\alpha}
\end{equation} with $k\alpha<n+2\delta$. Then for each $0<R<1/4$,\begin{equation}
\int_{B_{R}^+(x_0)}|Du-(Du)_R^+|^2\leq C'R^{k\alpha}.
\end{equation}
\end{lemma}

\begin{proof}
It is similar to Proposition $2.2.$ in \cite{gg}. Note that we have extra terms with $R^{n+2\Gamma}$ in \eqref{iter}, \eqref{iter1}. But those terms can be dominated by $CR^{k\alpha}$. 
\end{proof}

We choose (after decreasing $\alpha>0$ a little bit if needed) an $m\in\mathbb{N}$ so that $$(m-1)\alpha<n<m\alpha<n+2\delta.$$ Then we use induction and Companato space argument as in \cite{gg} to conclude $u\in C^{1,\sigma}(\overline{B_{3/4}^+})$ with $\sigma=(m\alpha-n)/2$.\\
Now we consider \eqref{de}\begin{equation}
-div\big(\mathbf{A}_{p_i}(\cdot,Dv)D_iw\big)=0\ \text{in }B_R^+(x_0).
\end{equation}
Using Holder continuity of $Dv$ one has \begin{equation}
\int_{B_\rho^+(x_0)}|Dw|^2\leq C\bigg(\frac{\rho}{R}\bigg)^{n-\varepsilon}\int_{B_R^+(x_0)}|Dw|^2. 
\end{equation}
Proceeding as before we have for given $\varepsilon>0$, \begin{equation}
\Phi^+(x_0,\rho)\leq C\bigg(\frac{\rho}{R}\bigg)^{n+2-\varepsilon}\Phi^+(x_0,R)+CR^{n+2\Gamma}\n f\n_{L^q(B_{R}^+(x_0))}^2+CR^{2\alpha}\int_{B_{R}^+(x_0)}(1+|Du|^2).
\end{equation}Using boundedness of $Du$ we have as $\alpha\leq\Gamma$,
\begin{equation}
\Phi^+(x_0,\rho)\leq C\bigg(\frac{\rho}{R}\bigg)^{n+2-\varepsilon}\Phi^+(x_0,R)+C(1+\n f\n_{L^q(B_{R}^+(x_0))}^2)R^{n+2\alpha}.
\end{equation}and so taking $\varepsilon>0$ small enough \begin{equation}
\Phi^+(x_0,\rho)\leq C\bigg(\frac{\rho}{R}\bigg)^{n+2\alpha}\big(\Phi^+(x_0,R)+(1+\n f\n_{L^q(B_{R}^+(x_0))}^2)R^{n+2\alpha}\big).
\end{equation}As before we conclude $u\in C^{1,\alpha}(\overline{B_{3/4}^+})$.
\end{proof}
\begin{remark}
Note that we are able to use the mercenary of \cite{gg} as in the last term of the inequality \eqref{dd}, the power of $R$ is atleast $n+2\alpha$. 
\end{remark}

\section{Quadratic Conormal Case}\label{sc}
We now state the analogous  regularity  result in case of Neumann conditions. 

\begin{theorem}\label{C}
Let $u\in W^{1,2}(B_1^+)\cup L^\infty(B_1^+)$ solves
\begin{equation}\label{lc}
-div\big(\mathbf{A}(\cdot,\nabla u)\big)=-div\ \mathbf{h}+f\ \text{in }B_1^+,\ \ \ \mathbf{A}^n(\cdot,\nabla u)=\phi\text{ on }B_1'. 
\end{equation}
with $|\partial_{\zeta_j}\mathbf{A}^i|\leq L$, $\partial_{\zeta_j}\mathbf{A}^i\xi_i\xi_j\geq\Lambda|\xi|^2$, $|\mathbf{A}^i(x,\zeta)-\mathbf{A}^i(y,\zeta)|\leq L(1+|\zeta|)|x-y|^\alpha$, $\n f\n_{\li}\leq L$ such that $\alpha\leq\Gamma$, $\n \mathbf{h}\n_{C^\Gamma(B_1)}\leq L$, $\n\phi\n_{C^\alpha(B_1')}\leq L$. Then we have a $C=C(L,\Lambda,..)\geq0$ such that $\n u\n_{C^{1,\alpha}(\overline{B_{3/4}^+})}\leq C$.
\end{theorem}

In order to prove this result, we need certain preparatory lemmas which are as follows. 

\begin{lemma}\label{Cd}
Let $v\in W^{1,2}(B_R^+)\cup L^\infty(B_R^+)$ solves
\begin{equation}
-div\ a\ Dv=0\ \text{in }B_R^+,\ \ \ a^{n,i}D_iv=0\text{ on }B_R'. 
\end{equation}with $|a^{i,j}|\leq L$, $\langle a\ \xi,\xi\rangle\geq\Lambda|\xi|^2$, $|a(x)-a(y)|\leq L|x-y|^\alpha$. Then $v$ satisfies,
\begin{equation}
\int_{B_r^+}|Dv|^2\leq C\frac{1}{(R-r)^2}\int_{B_{R}^+}(v-\lambda)^2.
\end{equation}and moreover if $a$ is a constant matrix
\begin{equation}
\int_{B_r^+}|Dw|^2\leq C\frac{1}{(R-r)^2}\int_{B_{R}^+}(w-\lambda)^2.
\end{equation}where $w=D_kv$, for $k=1,2,\cdots, n-1$.
\end{lemma}
\begin{proof}
Let $\eta$ be a cut-off function on $B_R$ relative to $B_r$ i.e. $\eta\in C_0^\infty(B_R)$ and satisfies $$0\leq\eta\leq1,\ \ \eta=1\text{ in }B_r,\ \ |D\eta|\leq\frac{C}{R-r}.$$ 
Multiplying the equation for $v$ with $\eta^2(v-\lambda)$ and using integration by parts
$$0=-\int_{B_{R}^+}\eta^2(v-\lambda)div\ a\ Dv=\int_{B_{R}^+}[\eta^2Dv+2\eta(v-\lambda)D\eta]\ a\ Dv+\int_{\partial B_{R}^+}\eta^2(v-\lambda)a\ Dv\cdot\nu.$$
Using boundary condition and as $\eta\in C_0^\infty(B_R)$ we conclude the last term of above equation is zero. Hence using ellipticity condition
\begin{align*}
\int_{B_{R}^+}\eta^2|Dv|^2&\leq C\int_{B_{R}^+}\eta(v-\lambda)D\eta\cdot a\ Dv\\
&\leq C\varepsilon\int_{B_{R}^+}\eta^2|a\ Dv|^2+\frac{C}{\varepsilon}\int_{B_{R}^+}(v-\lambda)^2|D\eta|^2
\end{align*}Use boundedness of $a$, $D\eta$ and choose $\varepsilon>0$ appropriately to conclude the first result.\\
Since $a$ is constant matrix $w$ satisfies similar equation as $v$.
\end{proof}
\begin{lemma}\label{Ce}
Under the assumption of last Lemma \ref{Cd} $$\int_{B_{R/2}^+}|D^2v|^2\leq \frac{C}{R^4}\int_{B_{R}^+}v^2$$ if $a$ is a constant matrix.
\end{lemma}
\begin{proof}
We have $$0=\int_{B_{R}^+}\phi\ div\ a\ Dv=\int_{B_{R}^+}\phi\ a^{i,j}\ D_{ij}v.$$
As $a^{n,n}$ is fairly away from zero by ellipticity, it follows for $0<r<R/2$, $$\int_{B_{r}^+}|D_{nn}v|^2\leq C\sum_{k=1}^{n-1}\int_{B_{R/2}^+}|DD_kv|^2.$$
Hence $$\int_{B_{r}^+}|D^2v|^2\leq C\sum_{k=1}^{n-1}\int_{B_{R/2}^+}|DD_kv|^2\leq \frac{C}{R^2}\sum_{k=1}^{n-1}\int_{B_{3R/4}^+}|D_kv|^2\leq \frac{C}{R^4}\int_{B_R^+}v^2.$$
\end{proof}
\begin{corollary}Let $v$ be as in last Lemma \ref{Ce}, then for any non negative integer $k$, $$\n v\n_{H^k(B_{1/2}^+)}\leq C\n v\n_{L^2(B_1^+)}\text{ and }\n Dv\n_{H^k(B_{1/2}^+)}\leq C\n Dv\n_{L^2(B_1^+)}.$$
\end{corollary}

\begin{lemma}\label{Ca}
Under the assumption of last Lemma \ref{Ce} for $0<r<R$, $$\int_{B_r^+}|Dv|^2\leq C\bigg(\frac{r}{R}\bigg)^n\int_{B_{R}^+}|Dv|^2.$$ 
\end{lemma}
\begin{proof}
Enough to show for $R=1$. Choose $k>n/2$ and then, for $0<r<1/2$, $$\int_{B_r^+}|Dv|^2\leq Cr^n\sup_{B_{1/2}^+}|Dv|^2\leq Cr^n\n Dv\n_{H^k(B_{1/2}^+)}^2\leq Cr^n\n Dv\n_{L^2(B_1^+)}=Cr^n\int_{B_1^+}|Dv|^2.$$ 
For $1/2\leq r\leq1$, $$\int_{B_r^+}|Dv|^2\leq\int_{B_1^+}|Dv|^2\leq2^nr^n\int_{B_1^+}|Dv|^2.$$
\end{proof}
\begin{lemma}\label{Cc}
Let $w\in W^{1,2}(B_1^+)\cup L^\infty(B_1^+)$ solves
\begin{equation}\label{Cb}
-div\ a\ Dw=0\ \text{in }B_1^+,\ \ \ a^{n,i}D_iw=0\text{ on }B_1'. 
\end{equation}with $|a^{i,j}|\leq L$, $\langle a\ \xi,\xi\rangle\geq\Lambda|\xi|^2$, $|a(x)-a(y)|\leq L|x-y|^\alpha$. Then $w$ satisfies, for any $\varepsilon>0$ small enough, if $0<r<R\leq R_0$ (for some $R_0\leq1$),
\begin{equation}
\int_{B_r^+}|Dw|^2\leq C\bigg(\frac{r}{R}\bigg)^{n-\varepsilon}\int_{B_{R/2}^+}|Dw|^2.
\end{equation}
\end{lemma}
\begin{proof}
Let $h\in L^2(B_1^+,\rn)$. First study the problem \begin{equation}
-div\ a(0)\ Dv=-div\ h\ \text{in }B_R^+,\ \ v=w\text{ on }\partial B_R^+\smallsetminus B_R',\ \ a^{n,i}(0)D_iv=h_n\text{ on }B_1'. 
\end{equation}Consider the problems \begin{equation}
-div\ a(0)\ Dv_1=0\ \text{in }B_R^+,\ \ v_1=w\text{ on }\partial B_R^+\smallsetminus B_R',\ \ a^{n,i}(0)D_iv_1=0\text{ on }B_1'; 
\end{equation}
\begin{equation}
-div\ a(0)\ Dv_2=-div\ h\ \text{in }B_R^+,\ \ v_2=0\text{ on }\partial B_R^+\smallsetminus B_R',\ \ a^{n,i}(0)D_iv_2=h_n\text{ on }B_1'. 
\end{equation}
For existence of $v_1$, use \cite{lie2} or Theorem 2.3 in \cite{t}. Then by Lemma \ref{Ca}$$\int_{B_r^+}|Dv_1|^2\leq C\bigg(\frac{r}{R}\bigg)^n\int_{B_{R}^+}|Dv_1|^2.$$ 
On the the hand
$$\Lambda\int_{B_R^+}|Dv_2|^2\leq\int_{B_R^+}a(0)Dv_2\cdot Dv_2=\int_{B_R^+}h\ Dv_2.$$
Note that 
$$\int_{B_R^+}|h\ Dv_2|\leq C\epsilon'\int_{B_R^+}|Dv_2|^2+\frac{C}{\epsilon'}\int_{B_R^+}h^2.$$ 
Choose $\epsilon'>0$ to conclude $$\int_{B_R^+}|Dv_2|^2\leq C\int_{B_R^+}h^2.$$
Now \begin{align*}
\int_{B_r^+}|Dv|^2&\leq 2\int_{B_r^+}|Dv_1|^2+2\int_{B_r^+}|Dv_2|^2\\
&\leq C\bigg(\frac{r}{R}\bigg)^n\int_{B_{R}^+}|Dv_1|^2+2\int_{B_r^+}|Dv_2|^2\\
&\leq C\bigg(\frac{r}{R}\bigg)^n\int_{B_{R}^+}|Dv|^2+C\int_{B_r^+}|Dv_2|^2\\
&\leq C\bigg(\frac{r}{R}\bigg)^n\int_{B_{R}^+}|Dv|^2+C\int_{B_R^+}h^2.
\end{align*}
We can rewrite \eqref{Cb} as\begin{equation}
-div\ a(0)\ Dw=-div\ \big([a(0)-a]\ Dw\big)\ \text{in }B_1^+,\ \ \ a^{n,i}(0)D_iw=[a(0)-a]^{n,i}\ D_iw\text{ on }B_1'. 
\end{equation}
Note that in $B_{R}^+$, $$|[a(0)-a]\ Dw|^2\leq CR^{2\alpha}|Dw|^2$$
 then we have for $0<r<R$
$$\int_{B_r^+}|Dw|^2\leq C\bigg[\bigg(\frac{r}{R}\bigg)^n+R^{2\alpha}\bigg]\int_{B_{R}^+}|Dw|^2.$$
Let $0<\varepsilon<n$. By iteration lemma one has for $0<r<R\leq R_0$ 
$$\int_{B_r^+}|Dw|^2\leq C\bigg(\frac{r}{R}\bigg)^{n-\varepsilon}\int_{B_{R}^+}|Dw|^2$$ for some $R_0\leq1$.
\end{proof}

With Lemma \ref{Cd}, \ref{Cc} in hand, we now proceed to the proof of Theorem 5.1. 

\begin{proof} [Proof of Theorem 5.1]
Consider the PDE
\begin{equation}
-div\ \big(\mathbf{A}(0,\nabla v)\big)=0\text{ in }B_R^+,\ \ v=u\text{ on }\partial B_R^+\smallsetminus B_R',\  \ \mathbf{A}^n(0,\nabla v)=\phi(0)\text{ on }B_R'. 
\end{equation}For existence of $v$, use Theorem 7.1, Theorem 7.2 in \cite{t}, and the references therein. 
Note that by method of difference quotient one has $v\in W^{2,2}(B_r^+)$ for any $0<r<R$. For $k=1,2,\cdots,n-1$, the function $w=D_kv$ solves \begin{equation}\label{w}
-div\big(\mathbf{A}_{p_i}(0,Dv)D_iw\big)=0\ \text{in }B_R^+(x_0),\  \ \mathbf{A}_{p_i}^n(0,Dv)D_iw=0\text{ on }B_R'.
\end{equation}
Note that $Dv$ is Holder continuous. Then by Lemma \ref{Cc}, for $0<\rho\leq R/2\leq R_0$,
\begin{equation}
\int_{B_\rho^+(x_0)}|Dw|^2\leq C\bigg(\frac{\rho}{R}\bigg)^{n-2+2\delta}\int_{B_{R/2}^+(x_0)}|Dw|^2.
\end{equation}
Using the equation for $v$, we have that \begin{equation}
\int_{B_\rho^+(x_0)}|D^2v|^2\leq C\bigg(\frac{\rho}{R}\bigg)^{n-2+2\delta}\sum_{k=1}^{n-1}\int_{B_{R/2}^+(x_0)}|DD_kv|^2.
\end{equation}
Applying Caccioppoli's inequality, see Lemma \ref{Cd}, to \eqref{w} we conclude that \begin{equation}
\int_{B_\rho^+(x_0)}|D^2v|^2\leq CR^{-2}\bigg(\frac{\rho}{R}\bigg)^{n-2+2\delta}\sum_{k=1}^{n-1}\int_{B_{R}^+(x_0)}|D_kv-(D_kv)_R^+|^2.
\end{equation} Then  Poincare's inequality implies
\begin{equation}
\sum_{k=1}^{n}\int_{B_\rho^+(x_0)}|D_kv-(D_kv)_\rho^+|^2\leq C\bigg(\frac{\rho}{R}\bigg)^{n+2\delta}\sum_{k=1}^{n-1}\int_{B_{R}^+(x_0)}|D_kv-(D_kv)_R^+|^2.
\end{equation}Set $$\Phi^+(x_0,\rho)=\sum_{k=1}^n\int_{B_\rho^+(x_0)}|D_ku-(D_ku)_\rho^+|^2.$$ Then$$\Phi^+(x_0,\rho)\leq C\bigg(\frac{\rho}{R}\bigg)^{n+2\delta}\Phi^+(x_0,R)+C\int_{B_{R}^+(x_0)}|Du-Dv|^2.$$
Now
\begin{align*}
\int_{B_{R}^+}[\mathbf{A}&(0,Du)-\mathbf{A}(0,Dv)]\cdot D(u-v)=\int_{B_{R}^+}[\mathbf{A}(0,Du)-\mathbf{A}(x,Du)]\cdot D(u-v)\\
&\ +\int_{B_{R}^+}\mathbf{h}\cdot D(u-v)+\int_{B_{R}^+}(u-v)f-\int_{B_{R}'}(u-v)(\phi-\phi(0)-\mathbf{h}^n).
\end{align*}
Using divergence theorem as in section $4$ of \cite{lie}, taking $\mathbf{h}(0)=0$ and as $\alpha\leq\Gamma$,
\begin{align*}
\int_{B_{R}'}|u-v||\phi-\phi(0)-\mathbf{h}^n|&\leq CR^\alpha\int_{B_{R}'}|u-v|\\
&\leq CR^\alpha\int_{B_{R}^+}|Du-Dv|\\
&\leq C\varepsilon\int_{B_{R}^+}|Du-Dv|^2+C\varepsilon^{-1} R^{n+2\alpha}.
\end{align*}
For $0<\rho<R/2$, 
\begin{align*}
\Phi^+(x_0,\rho)\leq C\bigg(\frac{\rho}{R}\bigg)^{n+2\delta}\Phi^+(x_0,R)+C\big(R^{2\alpha}\n 1+|Du|\n_{L^2(B_{R}^+)}^2\\
+R^{n+2(1-n/q)}\n f\n_{L^q(B_{R}^+)}^{2}+R^{n+2\alpha}\big).
\end{align*}
Now we proceed as in Dirichlet case for rest of the proof.
\end{proof}

\section*{Acknowledgement}

I am thankful to Agnid Banerjee for various discussions and suggestions regarding this problem. I am also thankful to K. Sandeep for his guidance as my Ph.D. supervisor. I would also like to thank the referee for carefully reading the manuscript and for many incisive  comments and suggestions that immensely helped in improving the presentation of the manuscript.

\end{document}